\documentclass[12pt]{amsart}

\usepackage{amssymb}
\usepackage{bm}
\usepackage{graphicx}
\usepackage{psfrag}
\usepackage[usenames,dvipsnames]{xcolor}
\usepackage{enumerate}
\usepackage{multirow}
\usepackage{url}

\usepackage{comment}

\usepackage{algpseudocode}

\usepackage{mathtools}

\usepackage{blkarray}
\usepackage{tikz-cd}

\newcommand{\excise}[1]{}

\newtheorem{thm}{Theorem}[section]
\newtheorem{lemma}[thm]{Lemma}

\newtheorem{prop}[thm]{Proposition}

\newtheorem{question}[thm]{Question}

\theoremstyle{definition}

\newtheorem{example}[thm]{Example}

\numberwithin{equation}{section}

\renewcommand\>{\rangle}
\newcommand\<{\langle}

\newcommand\RR{\mathbb{R}}

\newcommand\ZZ{\mathbb{Z}}
\newcommand\dd{{\mathbf d}}

\newcommand\xx{{\mathbf x}}

\newcommand\zz{{\mathbf z}}

\DeclareMathOperator\lcm{lcm} 
\DeclareMathOperator\Ap{Ap} 

\begin{document}

\mbox{}
\title[Some asymptotic results on $p$-lengths of factorizations]{Some asymptotic results on $p$-lengths of factorizations for numerical semigroups \\ and arithmetical congruence monoids}

\author[S.~Chapman]{Spencer Chapman}
\address{Mathematics Department\\Trinity University\\San Antonio, TX 78212}
\email{schapma3@trinity.edu}

\author[E.~B.~Dugan]{Eli B.~Dugan}
\address{Mathematics and Statistics Department\\Williams College\\Williamstown, MA 01267}
\email{ebd3@williams.edu}

\author[S.~Gaskari]{Shadi Gaskari}
\address{Mathematics Department\\San Diego State University\\San Diego, CA 92182}
\email{shgaskari@gmail.com}

\author[E.~Lycan]{Emi Lycan}
\address{Mathematics Department\\San Diego State University\\San Diego, CA 92182}
\email{rlycan0041@sdsu.edu}

\author[S.~Mendoza]{Sarah Mendoza De La Cruz}
\address{Mathematics Department\\University of Texas at Austin\\Austin, TX 78712}
\email{sarahpaolam11@gmail.com}

\author[C.~O'Neill]{Christopher O'Neill}
\address{Mathematics Department\\San Diego State University\\San Diego, CA 92182}
\email{cdoneill@sdsu.edu}

\author[V.~Ponomarenko]{Vadim Ponomarenko}
\address{Mathematics Department\\San Diego State University\\San Diego, CA 92182}
\email{vponomarenko@sdsu.edu}

\date{\today}

\begin{abstract}
A factorization of an element $x$ in a monoid $(M, \cdot)$ is an expression of the form $x = u_1^{z_1} \cdots u_k^{z_k}$ for irreducible elements $u_1, \ldots, u_k \in M$, and the length of such a factorization is $z_1 + \cdots + z_k$.  We introduce the notion of $p$-length, a generalized notion of factorization length obtained from the $\ell_p$-norm of the sequence $(z_1, \ldots, z_k)$, and present asymptotic results on extremal $p$-lengths of factorizations for large elements of numerical semigroups (additive submonoids of $\ZZ_{\ge 0}$) and arithmetical congruence monoids (certain multiplicative submonoids of $\ZZ_{\ge 1}$).  Our results, inspired by analogous results for classical factorization length, demonstrate the types of combinatorial statements one may hope to obtain for sufficiently nice monoids, as well as the subtlety such asymptotic questions can have for general monoids.  
\end{abstract}

\maketitle


\section{Introduction}
\label{sec:intro}

Given a cancellative, commutative monoid $(M, \cdot)$, a \emph{factorization} of $x \in M$ is an expression of the form
\begin{equation}\label{eq:firstfactorization}
x = u_1^{z_1} \cdots u_k^{z_k}
\end{equation}
where $u_1, \ldots, u_k \in M$ are distinct irreducible elements (or \emph{atoms}) and $\zz \in \ZZ_{\ge 1}^k$.  One~of the primary goals of factorization theory is to characterize and quantify the non-uniqueness of factorizations of monoid elements~\cite{nonuniq}.  To this end, one of the predominant invariants examined is a factorization's \emph{length}, which coincides with the 1-norm $z_1 + \cdots + z_k$ of $\zz$.  A cornerstone of their study is the so-called structure theorem for sets of length, which has been shown to hold for large families of monoids and provides a combinatorial description of the set of possible factorization lengths of elements of the form $x^n$ for large $n$; see~\cite{setsoflengthmonthly} for a survey of such results.  

In this manuscript, we consider two particular families of monoids.  The first, known as \emph{arithmetical congruence monoids}, are multiplicative submonoids of $\ZZ_{\ge 1}$ of the form
$$M_{a,b} = \{1\} \cup \{n \in \ZZ_{\ge 1} : n \equiv a \bmod b\}$$
for $a,b \ge 1$ satisfying $a^2 \equiv a \bmod b$.  For instance, in 
$$M_{1,4} = \{1, 5, 9, 13, \cdots\},$$
the elements 9, 21, and 49 are all atoms, and consequently $441 = 9 \cdot 49 = 21^2$ admits two distinct factorizations (this particular monoid is known as the Hilbert monoid, as he used it to exhibit non-unique factorization).  Since their introduction in~\cite{acmarithmetic}, arithmetical congruence monoids (ACMs for short) have garnered some attention in the factorization theory community as a source of ``naturally-occuring'' monoids with a surprising propensity to exhibit pathological factorization behavior; see the survey~\cite{acmsurvey} for an overview of the known factorization properties of ACMs and a number of lingering questions.  

The second family of monoids are \emph{numerical semigroups}, which are additive submonoids of $\ZZ_{\ge 0}$ with finite complement.  Any numerical semigroup $S$ has finitely many atoms $g_1, \ldots, g_k$, in which case we write
$$S = \<g_1, \ldots, g_k\> = \{z_1g_1 + \cdots + z_kg_k : z_1, \ldots, z_k \in \ZZ_{\ge 0}\}.$$
Note that, unlike ACMs, the operation on a numerical semigroup is \textbf{addition}; as such, a factorization of $n \in S$ is an additive expression of the form 
$$n = z_1g_1 + \cdots + z_kg_k$$
for some $\zz \in \ZZ_{\ge 0}^k$.  
Numerical semigroups have long arisen in countless setting across the mathematical spectrum, including combinatorics, algebra, number theory, polyhedral geometry, and discrete optimization; we direct the reader to the monographs~\cite{numericalappl,numerical} for a thorough introduction to numerical semigroups.  


In this manuscript, we consider notions of factorization length derived from $\ell_p$-norms for other values of $p \in \ZZ_{\ge 0} \cup \{\infty\}$.  More specifically, given a general monoid $M$ (written multiplicatively), the \emph{$p$-length} of the factorization in~\eqref{eq:firstfactorization} is given by 
$$\ell_p(\zz) = z_1^p + \cdots + z_k^p$$
if $p \in \ZZ_{\ge 0}$, which if $p > 0$ coincides with $(\|\zz\|_p)^p$, and 
$$\ell_\infty(\zz) = \|\zz\|_\infty = \max(z_1, \ldots, z_k).$$
In the case $p = 1$, one obtains the usual notion of factorization length.  

We focus our attention on the extremal $p$-lengths of factorizations for large monoid elements, drawing inspiration from the results of~\cite{asympminmaxlength} that, under minimal assumptions on the monoid $M$ and the element $x \in M$, the limits
$$
\lim_{n \to \infty} \frac{\sup L(x)}{n}
\qquad \text{and} \qquad
\lim_{n \to \infty} \frac{\min L(x)}{n}
$$
both exist, though the former need not be finite (here, $L(x)$ denotes the set of (classical) lengths of factorizations of $x$).  
Our study focuses on numerical semigroups (Section~\ref{sec:ns}), for which we derive asymptotic results of a combinatorial nature that are familiar in this setting~\cite{numericalsurvey}, and arithmetical congruence monoids (Section~\ref{sec:acms}), wherein our results demonstrate the subtlety of such asymptotic questions for more general monoids.





\section{Numerical semigroups}
\label{sec:ns}

Fix a numerical semigroup $S = \<g_1, \ldots, g_k\>$.  Given $\zz \in \ZZ^k$, define
$$
\ell_\infty(\zz) = \max\{z_1, \ldots, z_k\}
\quad \text{and} \quad
\ell_p(\zz) = z_1^p + \cdots + z_k^p
\quad \text{for} \quad
p \in \ZZ_{\ge 0},
$$
and for $p \in \ZZ_{\ge 0} \cup \{\infty\}$, define
$$
\ell_p^M(n) = \max\{\ell_p(\zz) : \zz \in \mathsf Z(n)\}
\quad \text{and} \quad
\ell_p^m(n) = \min\{\ell_p(\zz) : \zz \in \mathsf Z(n)\}
$$
for each $n \in S$, where
$$\mathsf Z(n) = \{\zz \in \ZZ_{\ge 0}^k : n = z_1 n_1 + \cdots + z_k n_k\}$$
is the \emph{set of factorizations} of $n$ in $S$.  

For a fixed numerical semigroup, the asymptotic behavior of $\ell_p^m(n)$ and $\ell_p^M(n)$ often take a particularly nice combinatorial form.  We say a function $f:\ZZ_{\ge 0} \to \RR$ is a \emph{quasipolynomial} if there exist periodic functions $c_d(n)$, $c_{d-1}(n)$, \ldots, $c_0(n)$, with $c_d$ not identically 0, such that 
$$f(n) = c_d(n) n^d + \cdots + c_1(n)n + c_0(n);$$
in this case, the \emph{degree} of $f$ is $d$ and the \emph{period} of $f$ is the least common multiple of the periods of the $c_i(n)$.  
We say $f$ is \emph{eventually quasipolynomial} if there exists a quasipolynomial $g(n)$ such that $f(n) = g(n)$ for all $n \gg 0$ (that is, for all but finitely many $n \in \ZZ_{\ge 0}$).  

Quasipolynomial functions often arise in the context of numerical semigroups; see~\cite{numericalsurvey} for a survey of such results, and~\cite{continuousdiscretely} for geometric interpretations of this phenomenon.  Theorems~\ref{t:l0andl1}, \ref{t:maxinf}, \ref{t:mininf}, \ref{t:maxp}, and~\ref{t:min2} imply each of the functions in Table~\ref{tb:eqpsummary} are eventually quasipolynomial functions of $n$ with the specified degree, period, and constant leading coefficient.  

\begin{table}[t]
\begin{center}
$\begin{array}{|l|l|l|l|}
\text{function} & \text{degree} & \text{period} & \text{leading coefficient}
\\ \hline
\ell_0^m(n)
& 0 & \lcm(g_1, \ldots, g_k) & 
\\[0.1em]
\ell_1^m(n)
& 1 & g_k & 1/g_k
\\[0.1em]
\ell_2^m(n)
& 2 & g_1^2 + \cdots + g_k^2 & 1/(g_1^2 + \cdots + g_k^2)
\\[0.1em]
\ell_\infty^m(n)
& 1 & g_1 + \cdots + g_k & 1/(g_1 + \cdots + g_k)
\\[0.2em] \hline
\ell_0^M(n)
& 0 & 1 & k
\\[0.1em]
\ell_{p \ge 1}^M(n)
& p & g_1 & 1/g_1
\\[0.2em]
\ell_\infty^M(n)
& 1 & g_1 & 1/g_1
\\ \hline
\end{array}$
\end{center}
\caption{The eventually quasipolynomial attributes of $\ell_p^m(n)$ and $\ell_p^M(n)$ for $p \in [0, \infty]$ over a numerical semigroup $S = \<g_1, \ldots, g_k\>$.}
\label{tb:eqpsummary}
\end{table}

Results for $p = 0$ and $p = 1$ appeared in~\cite{quasi0norm,elastsets}; we state them here for completeness.  In what follows, let 
$$F(S) = \max(\ZZ_{\ge 0} \setminus S)$$
denote the \emph{Frobenius number} of $S$.  

\begin{thm}\label{t:l0andl1}
Fix a numerical semigroup $S = \<g_1, \ldots, g_k\>$.  
\begin{enumerate}[(a)]
\item 
For $n > (g_1 - 1)g_k$, we have
$$
\ell_1^m(n) = \ell_1^m(n - g_k) + 1
$$
As such, $\ell_1^m(n)$ is eventually quasilinear with period $g_k$ and leading coefficient $\tfrac{1}{g_k}$.  

\item 
For $n > (g_{k-1} - 1)g_k$, we have
$$
\ell_1^M(n) = \ell_1^M(n - g_1) + 1
$$
As such, $\ell_1^M(n)$ is eventually quasilinear with period $g_1$ and leading coefficient $\tfrac{1}{g_1}$.  

\item 
For $n > g_k^2$, $\ell_0^m(n)$ is periodic with period $\lcm(g_1, \ldots, g_k)$.  

\item 
For $n > F(S) + g_1 + \cdots + g_k$, we have $\ell_0^M(n) = k$.  

\end{enumerate}
\end{thm}

\begin{proof}
Parts~(a) and~(b) follow from \cite[Theorem~4.3]{elastsets} and \cite[Theorem~4.2]{elastsets}, respectively.  Additionally, part~(c) follows from \cite[Theorem~12]{quasi0norm} and the bounds on $F(S)$ in~\cite{diophantinefrob}.  Lastly, part~(d) follows from the fact that every specified values of $n$ has a factorization involving all of the generators of $S$.  
\end{proof}

Next, we consider $p = \infty$.  The results of~\cite{structurethmns} strengthen Theorem~\ref{t:l0andl1}(a) and~(b) to include an interpretation of the values the periodic constant coefficient takes.  Analogously, in addition to proving $\ell_\infty^m(n)$ and $\ell_\infty^M(n)$ are eventually quasipolynomial and determining their leading coefficients, we identify an interpretation of the values taken by their respective periodic constant terms.  

The \emph{Ap\'ery set} of a numerical semigroup $S$ with respect to an element $m \in S$ is 
$$\Ap(S; n) = \{n \in S : n - m \notin S\}.$$
It is known that $\Ap(S;n)$ is comprised of the smallest element of each equivalence class modulo $n$, that is,
$\Ap(S; n) = \{0, a_1, \ldots, a_{n-1}\},$
where $a_i$ is the smallest element of $S$ with the property $a_i \equiv i \bmod n$ (see~\cite{numericalappl}).  

In what follows, let
$$g = g_1 + \cdots + g_k.$$

\begin{lemma}\label{l:infmin}
For any $n \in S$ and $c \in \ZZ_{\ge 1}$, if $n > cg$, then $\ell_\infty^m(n) > c$.
\end{lemma} 

\begin{proof}
Suppose $\ell_\infty^m(n) \le c$.  Some factorization $\zz \in \mathsf Z(n)$ must have $z_i \le c$ for all $i$, so 
$$n = z_1g_1 + \cdots + z_kg_k \le cg_1 + cg_2 + \cdots + cg_k = cg.$$
The claim now follows.  
\end{proof}

\begin{lemma}\label{l:infmax}
Fix $n \in S$ with $n > g_1^2 g$.  If a factorization $\zz \in \mathsf Z(n)$ has $\ell_\infty(\zz) = z_i$ for $i \in \{2, 3, \cdots, k\}$, then there exists a factorization $\zz' \in \mathsf Z(n)$ such that $\ell_\infty(\zz') > \ell_\infty(\zz)$. 
\end{lemma}

\begin{proof}
We have $n > g_1^2 g$, so by Lemma~\ref{l:infmin}, $\ell(\zz) = z_i > g_1^2$.  Write $z_i = qg_1 + r$ for $q, r \in \ZZ$ with $0 \leq r < g_1$, so in particular $q \ge g_1$.  
Trading $qg_1$ copies of $g_i$ for $qg_i$ copies of $g_1$, we obtain a factorization $\zz' \in \mathsf Z(n)$ with $z_1' = z_1 + qg_i$, $z_i' = z_i - qg_1 = r \ge 0$, and $z_j' = z_j$ for all other $j$.  
We then readily check
$$
z_1 + qg_i
\ge qg_i
\ge q(g_1 + 1)
= qg_1 + q
\ge qg_1 + g_1
> qg_1 + r
= z_i,
$$
so $z_1' > z_i$, and therefore $\ell_\infty(\zz') \ge z_1' > z_i = \ell_\infty(\zz)$.
\end{proof}



\begin{thm}\label{t:maxinf}
Write $\Ap(S; g_1) = \{a_0, a_1, a_2, \ldots\}$ where each $a_j \equiv j \bmod g_1$. 
For all $n \in S$ with $n > g_1^2 g$ and $n \equiv i \bmod g_1$, we have
\[
\ell_\infty^M(n) = \tfrac{1}{g_1}(n - a_i).
\]
In particular, for all $n > g_1^2 g$, we have
\[
\ell_\infty^M(n) = \ell_\infty^M(n - g_1) + 1.
\]
\end{thm}

\begin{proof}
Suppose $\zz \in \mathsf Z(n)$ has maximal $\infty$-norm.  Since $n > g_1^2 g$, Lemma~\ref{l:infmax} implies
$$\ell_\infty^M(n) = \max \{z_1' : \zz' \in \mathsf Z(n)\}.$$
and in particular  $\ell_\infty(\zz) = z_1 \ge z_1'$ for all $\zz' \in \mathsf Z(n)$.  This means $n - z_1g_1 \in \Ap(S;g_1)$.  Comparing equivalence classes modulo $g_1$, this means $n - z_1g_1 = a_i$, and solving yields
$$\ell_\infty^M(n) = \tfrac{1}{g_1}(n - a_i).$$
The final claim now immediately follows.  
\end{proof}

\begin{lemma}\label{l:mininfapery}
If $a \in \Ap(S;g)$, then $\ell_\infty^m(a) < g$.
\end{lemma}


\begin{proof}
Suppose $\ell_\infty^m(a) \ge g$. Then there is a factorization $\zz \in \mathsf Z(a)$ such that $z_i \ge g$ for some $i$.  This means $\zz' \in \mathsf Z(a)$ given by
$$
z_j' = \begin{cases}
z_j + g_i - g & \text{if } j = i;
\\
z_j + g_i & \text{otherwise,}
\end{cases}
$$
is a factorization of $a$ since
$$
(z_i + g_i - g)g_i + \sum_{j \ne i} (z_j + g_i) g_j
= -g_i g + \sum_{i = 1}^k (z_j + g_i)g_j
= -g_i g + a + g_i g
= a.
$$
Since every coordinate of $\zz'$ is positive, $a - g \in S$, and thus $a \notin \Ap(S;g)$.  
\end{proof}

 
\begin{thm}\label{t:mininf}
Write $\Ap(S; g) = \{0, a_1, \ldots, a_{g-1}\}$ with each $a_j \equiv j \bmod g$.  Fix $n \in S$, and fix $i \in \{0, 1, \cdots, g-1\}$ so that $i \equiv -n \bmod g$. Then 
\[\ell_\infty^m(n) = \tfrac{1}{g}(n + a_i)\]
for all $n > g^2$.  In particular, for all $n > g^2$, we have
$$\ell_\infty^m(n) = \ell_\infty^m(n - g) + 1.$$
\end{thm}

\begin{proof}
Fix $q \in \ZZ$ so that $n = qg - a_i$.  We claim $\ell_\infty^m(n) = q$.  

We then get that $\ell_\infty^m(n + a_i) = l_\infty(qg) = q$, which is achieved by the factorization $(q, q, \ldots, q) \in \mathsf Z(qg)$.  Now, any factorization $\zz' \in \mathsf Z(a_i)$ has $z_j < g$ for each $j$ by Lemma~\ref{l:mininfapery}, and since $n > g^2$, we must have $g < q$, so
$$(q, q, \ldots, q) - \zz' \in \mathsf Z(qg - a_i) = \mathsf Z(n),$$
and in particular $\ell_\infty^m(n) \le q$.


Next, suppose by way of contradiction that $\ell_\infty^m(n) < q$.  Then there is a factorization $\zz \in \mathsf Z(n)$ with $z_j < q$ for every $j$.  However, this implies
$$(q, q, \ldots, q) - \zz \in \mathsf Z(qg - n) = \mathsf Z(a_i)$$
which is impossible since this factorization has no nonzero entries and $a_i \in \Ap(S;g)$.  As such, we conclude 
$$\ell_\infty^m(n) = q = \tfrac{1}{g}(n + a_i),$$
from which the final claimed equality immediately follows.  
\end{proof}

Lastly, we next turn our attention to $2 \le p < \infty$, for which the asymptotic form of $\ell_p^M(n)$ is again quasipolynomial, while the asymptotic form of $\ell_p^m(n)$ is more nuanced.  

\begin{lemma}\label{l:maxp}
Fix $p \in \ZZ_{\ge 2}$.  For all $n \gg 0$, $\ell_p^M(n)$ is achieved by a factorization of $n$ with maximal first coordinate.  
\end{lemma}

\begin{proof}
Fix $a \in \Ap(S;g_1)$, and consider $n \in S$ with $n \equiv a \bmod g_1$.  Writing $n = qg_1 + a$, it is the case that $q$ is the largest first coordinate occuring in any factorization of $n$.  Fix any factorization $\zz \in \mathsf Z(n)$ with $z_1 = q$, and fix $\zz' \in \mathsf Z(n)$ with $z_1' < q$.  We seek to prove that for $q \gg 0$, we have $\ell_p(\zz) \ge \ell_p(\zz')$.  

A simple calculus exercise verifies that the maximum of $x_2^p + \cdots + x_k^p$ for $\xx \in \RR_{\ge 0}^{k-1}$ subject to the constraint $x_2g_2 + \cdots + x_kg_k = n$ occurs when $\xx = (\tfrac{1}{g_2}n, 0, \ldots, 0)$.  As~such, letting $c = z_1' - q$ and noting that $(z_2', \ldots, z_k')$ is a factorization of $n - z_1'g_1$ in the semgiroup $\<g_2, \ldots, g_k\>$, we see 
\begin{equation}\label{eq:maxpbound}
\ell_p(\zz')
\le (q - c)^p + (\tfrac{1}{g_2}(n - z_1'g_1))^p
= (q - c)^p + (\tfrac{1}{g_2}(a + cg_1))^p.
\end{equation}
Now, for $q$ sufficiently large, 
$$\ell_p(\zz) \ge q^p \ge (q - 1)^p + (\tfrac{a}{g_2} + \tfrac{g_1}{g_2})^p$$
since $p \ge 2$ and $(\tfrac{a}{g_2} + \tfrac{g_1}{g_2})^p$ is constant.  As such, to complete the proof, it suffices to show that the right hand side of~\eqref{eq:maxpbound} is maximized over real $c \in [1, q]$ when $c = 1$.  Again using methods from calculus, there is a unique local extremum in $[1, q]$, and it is a local minimum, so the maximum value must be attained at either $c = 1$ or $c = q$.  Indeed, for $q$ sufficiently large, we obtain
$$(q - 1)^p + (\tfrac{a}{g_2} + \tfrac{g_1}{g_2})^p \ge (\tfrac{a}{g_2} + q\tfrac{g_1}{g_2})^p$$
as $g_2 > g_1$ ensures $(q - 1)^p - (\tfrac{a}{g_2} + q\tfrac{g_1}{g_2})^p$ eventually surpasses the constant $(\tfrac{a}{g_2} + \tfrac{g_1}{g_2})^p$.  
\end{proof}

\begin{thm}\label{t:maxp}
If $p \in \ZZ_{\ge 2}$, then $\ell_p^M(n)$ is eventually quasipolynomial with degree $p$, period $g_1$, and constant leading coefficient $1/g_1^p$.  
\end{thm}

\begin{proof}
Let $f(n)$ denote the largest first coordinate of any factorization in $\mathsf Z(n)$, and 
$$g(n) = \ell_p^M \big( n - f(n)g_1 \big).$$
Since $n - f(n)g_1 \in \Ap(S; g_1)$ by definition, $g(n)$ is periodic with period $g_1$, and $f(n)$ is eventually quasilinear with period $g_1$ and leading coefficient $1/g_1$.  Since $\ell_p^M(n)$ is achieved by a factorization with maximal first coordinate for $n \gg 0$ by Lemma~\ref{l:maxp}, 
$$\ell_p^M(n) = f(n)^p + g(n)$$
is quasipolynomial of degree $p$, period $g_1$, and leading coefficient $1/g_1^p$.  
\end{proof}

We now examine $\ell_2^m(n)$.  In what follows, let 
$$N = g_1^2 + \cdots + g_k^2.$$

\begin{prop}\label{p:m2solutions}
Fix $n \ge 0$.  If $\zz \in \ZZ^k$ minimizes $\ell_2(\cdot)$ among all integer solutions to 
\begin{equation}\label{eq:one}
x_1g_1 + \cdots + x_kg_k = n,
\end{equation}
then $\zz + (g_1, \ldots, g_k)$ minimizes $\ell_2(\cdot)$ among all integer solutions to 
\begin{equation}\label{eq:two}
x_1g_1 + \cdots + x_kg_k = n + N.
\end{equation}
\end{prop}

\begin{proof}
A solution $\zz \in \ZZ^k$ to~\eqref{eq:one} minimizes $\ell_2(\cdot)$ if and only if
$$\ell_2(\zz') - \ell_2(\zz) = \sum_{i = 1}^k (z_i')^2 - \sum_{i = 1}^k z_i^2 = \sum_{i = 1}^k 2z_id_i + d_i^2 \ge 0$$
for every other solution $\zz' \in \ZZ^k$ to~\eqref{eq:one}, where $\dd = \zz' - \zz$.  In particular, this holds if and only if 
$$\sum_{i = 1}^k 2z_id_i + d_i^2 \ge 0$$
for all $\dd \in \ZZ^k$ satiftying $d_1g_1 + \cdots + d_kg_k = 0$.

Suppose $\zz$ satisfies this property.  For all $\dd$ satisfying $d_1g_1 + \cdots + d_kg_k = 0$,
$$\sum_{i = 1}^k 2(z_i + g_i)d_i + d_i^2 = 2\sum_{i = 1}^k g_id_i + \sum_{i = 1}^k 2z_id_i + d_i^2 \ge 0,$$
which implies that $\zz + (g_1, \ldots, g_k)$ minimizes $\ell_2(\cdot)$ among solutions to~\eqref{eq:two}.
\end{proof}

\begin{thm}\label{t:min2}
The function $\ell_2^m(n)$ is a quasipolynomial of degree 2 on $n \gg 0$, with period $N$ and constant leading coefficient $1/N$.  
\end{thm}

\begin{proof}
For $n \in \ZZ_{\ge 0}$, Proposition~\ref{p:m2solutions} implies the smallest coordinate of the integer solution to~\eqref{eq:two} minimizing $\ell_2(\cdot)$ is strictly larger than the smallest coordinate of the integer solution to~\eqref{eq:two} minimizing $\ell_2(\cdot)$.  As such, one may choose $n \in S$ large enough to ensure some $\zz \in \mathsf Z(n)$ minimizes $\ell_2(\cdot)$ over all integer solutions.  We then have
$$\ell_2^m(n + N) - \ell_2^m(n) = \sum_{i = 1}^k 2z_ig_i + g_i^2 = 2n + N,$$
and taking second differences yields
$$(\ell_2^m(n + 2N) - \ell_2^m(n + N)) - (\ell_2^m(n + N) - \ell_2^m(n)) = (2(n + N) + N) - (2n + N) = 2N,$$
which is independent of $n$.  The result follows.
\end{proof}

\begin{example}\label{e:min3}
It turns out $\ell_p^m(n)$ is not necessarily an eventual quasipolynomial for $3 \le p < \infty$.  For example, if $S = \<2,3\>$, then by a similar computation to the one used in the proof of Proposition~\ref{p:m2solutions}, $\mathsf m_3(n)$ is achieved by the solution $(c_1,c_2) \in \mathsf Z(n)$ with
$$c_1 = \left\lfloor\frac{-8 \pm n\sqrt{130}}{19}\right\rfloor.$$
This produces a formula for $\ell_3^m(n)$ involving the floor of an integer multiple of an irrational number, which cannot be eventually quasipolynomial.  
\end{example}

\section{Arithmetical congruence monoids}
\label{sec:acms}

In this section, we turn our attention to arithmetical congruence monoids.  
For a fixed ACM $M_{a,b}$, write
$$\mathsf Z(x) = \{\zz \in \ZZ_{\ge 0}^k : x = u_1^{z_1} \cdots u_k^{z_k} \text{ for some $k \ge 0$ and distinct atoms } u_1, \ldots, u_k \in M_{a,b}\}$$
for the set of \emph{factorization tuples} of $x$, 
for each $p \in \ZZ_{\ge 0} \cup \{\infty\}$ define
$$
\ell_p^M(x) = \max\{\ell_p(\zz) : \zz \in \mathsf Z(x)\}
\quad \text{and} \quad
\ell_p^m(x) = \min\{\ell_p(\zz) : \zz \in \mathsf Z(x)\}
$$
for each $x \in M_{a,b}$.  

Given the aperiodic nature of primes in $\ZZ$, one would expect that the above functions will not be quasipolynomial like their analogues for numerical semigroups.  As such, our results in this section seek only to determine asymptotic growth rate.  For each $p \in \{0,1,\infty\}$, the asymptotics of one of the two functions is straightforward to identify.  

\begin{thm}\label{t:acmeasy}
For each fixed $x \in M_{a,b}$ with $x > 1$, we have 
$$
\ell_1^M(x^n) \in \Theta(n),
\qquad \ell_\infty^M(x^n) \in \Theta(n), 
\qquad \text{and} \qquad
\ell_0^m(x^n) \in \Theta(1).
$$
\end{thm}

\begin{proof}
Fix a factorization $x = u_1^{z_1} \cdots u_k^{z_k}$, where $u_1, \ldots, u_k \in M_{a,b}$ are distinct irreducible elements, and let $k'$ denote the number of primes in the usual factorization of $x \in \ZZ$.  Since
$$x^n = u_1^{nz_1} \cdots u_k^{nz_k}$$
for each $n \ge 1$, we see 
$$
n \le \ell_\infty^M(x^n) \le \ell_1^M(x^n) \le k'n
\qquad \text{and} \qquad
\ell_0^m(x^n) \le k.
$$
All 3 claims follow from the above bounds.  
\end{proof}

If $a = 1$, the ACM $M_{a,b}$ is said to be \emph{regular}; in this case, it is know that for fixed $x \in M_{a,b}$ with $x > 1$, the set of irreducible elements dividing any power $x^n$ is finite (this follows from the fact that $M_{a,b}$ is Krull with finite class group; we refer the reader to~\cite[Section~3]{acmsurvey} for details).  We record the following consequences of this fact.  

\begin{thm}
If $a = 1$, then for each fixed $x \in M_{a,b}$ with $x > 1$, we have
$$
\ell_1^m(x^n) \in \Theta(n),
\qquad \ell_\infty^m(x^n) \in \Theta(n), 
\qquad \text{and} \qquad
\ell_0^M(x^n) \in \Theta(1).
$$
\end{thm}

In contrast to Theorem~\ref{t:acmeasy} and the results of Section~\ref{sec:ns}, for \emph{singular} (i.e., non-regular) ACMs the asymptotics of $\ell_1^m(x^n)$ and $\ell_\infty^m(x^n)$, viewed as functions of $n$, need not grow linearly in $n$, and $\ell_0^M(x^n)$ need not be bounded.  

\begin{example}\label{e:bifurcus}
The singular ACM $M_{6,6}$ is bifurcus, meaning every reducible element can be written as a product of just two atoms~\cite{bifurcusacm}, so 
$$\ell_\infty^m(x) \le \ell_1^m(x) \le 2$$
for all $x \in M_{6,6}$.  This implies $\ell_\infty^m(x^n) \in \Theta(1)$ and $\ell_1^m(x^n) \in \Theta(1)$ as functions of $n$.  
%
%
%
\end{example}

For the remainder of this section, we turn our attention to the singular ACM $M_{4,6}$ and elements of the form $x = 2^a 5^b 7^c$ for $a, b, c \in \ZZ_{\ge 0}$.  After recalling a known characterization of the relevant atoms of $M_{4,6}$, we demonstrate that the asymptotic growth rate of $\ell_\infty^m(x^n)$ and $\ell_1^m(x^n)$ matches that of regular ACMs, while the asymptotic behavior of $\ell_0^M(x^n)$ depends on $x$ even in this restrictive setting.  

\begin{lemma}\label{l:meyersonatoms}
An integer $u = 2^a 5^b 7^c$ lies in $M_{4,6}$ if and only if $a \ge 1$ and $a + b$ is even.  Moreover, $u$ is irreducible if and only if one of the following holds:
\begin{enumerate}[(i)]
\item 
$a = 2$ and $b = 0$; or

\item 
$a =1$ and $b$ is odd.  

\end{enumerate}


\end{lemma}

\begin{proof}
This follows from~\cite{acmarithmetic} and \cite[Example~4.10]{acmsurvey}, which each give a full characterization of the atoms of $M_{4,6}$.  
\end{proof}

\begin{thm}\label{t:inftyacm}
For fixed $x = 2^a 5^b 7^c \in M_{4,6}$ with $a \ge 1$, we have 
$$
\ell_\infty^m(x^n) \in \Theta(n)
\qquad \text{and} \qquad
\ell_1^m(x^n) \in \Theta(n).
$$
\end{thm}

\begin{proof}
In what follows, we say an atom $u = 2^p 5^q 7^r \in M_{4,6}$ is \emph{good} if $a(q+r) \leq 3p(b+c)$, and \emph{evil} otherwise.  Note that $M_{4,6}$ has only finitely many good atoms since $p \le 2$ and $b$ and $c$ are fixed.  Fix a factorization
$$x^n = u_1 \cdots u_k u_1' \cdots u_m'$$
of $x^n$ into (not necessarily distinct) good atoms $u_1, \ldots, u_k$ and evil atoms $u_1', \ldots, u_m'$.  We~claim $2k \ge m$.  Indeed, write each 
$$u_i = 2^{p_i} 5^{q_i} 7^{r_i}
\qquad \text{and} \qquad
u_j' = 2^{p_j'} 5^{q_j'} 7^{r_j'},$$
let $P = p_1 + \cdots + p_k$ and $P' = p_1' + \cdots + p_m$, and define $Q, Q', R$, and $R'$ analogously.  By the above factorization for $x^n$, 
$$
a(Q' + R')
\le a(Q + R + Q' + R')
= \tfrac{1}{n}a(b + c)
= (P + P')(b + c),
$$
and since each $u_j'$ is evil and each $p_i, p_j' \in \{1,2\}$, 
$$
3m(b + c)
\le 3P'(b + c)
< a(Q' + R')
\le (P + P')(b + c)
\le 2(k + m)(b + c),
$$
from which the inequality $m \le 2k$ follows.  

Having shown this, letting $G$ denote the number of good atoms in $M_{4,6}$, the pigeonhole principle ensures that any factorization of $x^n$ must have at least $\tfrac{1}{3G}n$ copies of some good atom, so 
$$\tfrac{1}{3G}n \le \ell_\infty^m(x^n) \le \ell_1^m(x^n) \le \ell_1^M(x^n)$$
ensures $\ell_\infty^m(x^n) \in \Theta(n)$ and $\ell_1^m(x^n) \in \Theta(n)$ by Theorem~\ref{t:acmeasy}.  
\end{proof}

\begin{prop}\label{p:acm0norm1/2}
If $x = 28 = 2^2 \cdot 7$ or $x = 40 = 2^3 \cdot 5$ in $M_{4,6}$, then $\ell_0^M(x^n) \in \Theta(n^{1/2})$.  
\end{prop}

\begin{proof}
First, consider $x = 28$.  Let $T_k = \binom{k+1}{2}$ denote the $k$-th triangular number.  We~claim if $n \in [T_k, T_{k+1})$ for $k \ge 2$, then $\ell_0^M(x^n) = k + 1$.  Indeed,
$$x^n = (2^2 7)^n = (2^2)^{n-k-1} (2^2 7) (2^2 7^2) \cdots (2^2 7^k) (2^2 7^{n-T_k}).$$
is a factorization by Lemma~\ref{l:meyersonatoms}, and the first $k+1$ atoms are distinct, so we obtain a lower bound $\ell_0^M(x^n) \ge k + 1$.  Similarly, by Lemma~\ref{l:meyersonatoms}, any atom dividing $x^n$ must be of the form $2^2 7^i$ for some $i \ge 0$.  If a factorization of $x^n$ had at least $k+2$ distinct atoms, then each such atom must have a distinct number of 7's in its prime factorization, and the resulting expression for $x^n$ would contain at least
$$0 + 1 + 2 + \cdots + k + (k+1) = T_{k+1}$$
copies of 7, which is impossible if $n < T_{k+1}$.  This ensures $\ell_0^M(x^n) = k + 1$.  

Next, consider $x = 40$.  We claim if $n \in [k^2, (k+1)^2)$, then $\ell_0^M(x^n) = k + 1$.  Indeed, by Lemma~\ref{l:meyersonatoms}, any atom dividing $x^n$ must either be 4 or of the form $2 \cdot 5^i$ for some odd $i \ge 1$.  If a factorization of $x^n$ had at least $k+2$ distinct atoms, then each such atom must have a distinct number of 5's in its prime factorization, and the resulting expression for $x^n$ would contain at least
$$0 + 1 + 3 + \cdots + (2k - 1) + (2k + 1) = (k+1)^2$$
copies of 5, which is impossible if $n < (k + 1)^2$.  This ensures $\ell_0^M(x^n) \le k + 1$.  To ensure equality, we verify that if $n \not\equiv k \bmod 2$, then 
$$x^n = (2^3 5)^n = (2^2)^{(n - k - 1)/2} (2 \cdot 5) (2 \cdot 5^3) (2 \cdot 5^5) \cdots (2 \cdot 5^{2k-1})(2 \cdot 5^{n-k^2})$$
is a factorization for $x^n$, while if $n \equiv k \bmod 2$, then 
$$x^n = (2^3 5)^n = (2^2)^{(n-k-2)/2} (2 \cdot 5)^2 (2 \cdot 5^3) (2 \cdot 5^5) \cdots (2 \cdot 5^{2k-1})(2 \cdot 5^{n-k^2-1})$$
is a factorization.  Each expression contains at least $k+1$ atoms, so $\ell_0^M(x^n) = k + 1$.  
\end{proof}

\begin{prop}\label{p:acm0norm2/3}
If $x = 70 = 2 \cdot 5 \cdot 7 \in M_{4,6}$, then $\ell_0^M(x^n) \in \Theta(n^{2/3})$. 
\end{prop}

\begin{proof}
Consider expressions of the form
$$x^n = c u_1 \cdots u_k u_1' \cdots u_m',$$
where $u_1, \ldots, u_k, u_1', \ldots, u_m'$ are distinct positive integers with $u_i = 2^2 7^{p_i}$ for each $i$ and $u_j' = 2 \cdot 5^{q_j} 7^{r_j}$ for each $j$, and $c$ is any positive integer.  Note that any factorization of $x^n$ in $M_{4,6}$ is an expression of the above form with 0-norm $k + m$, so any upper bound on $k + m$ in expressions of the above form is an upper bound for $\ell_0^M(x^n)$.  

Now, by a similar argument to the first half of the proof of Proposition~\ref{p:acm0norm1/2}, by counting the total number of 7's in $u_1 \cdots u_k$, we see $k \in O(n^{1/2})$.  Similarly, there are $a + 1$ possible values for each $u_j'$ with exactly $a$ total copies of 5 and 7, so if $m \ge \binom{a+1}{2}$, then examining the total number of 5's and 7's in $u_1' \cdots u_m'$, we see
$$\sum_{i = 1}^a i(i+1) \le q_1 + \cdots + q_m + r_1 + \cdots + r_m \le 2n$$
meaning $m \in O(n^{2/3})$.  We conclude $k + m \in O(n^{2/3})$, and thus $\ell_0^M(x^n) \in O(n^{2/3})$.  

Conversely, observe that for each even $k \ge 2$, we may choose $c$ appropriately so that
$$x^n = (2^2)^c \prod_{a=1}^k \prod_{i = 1}^a (2 \cdot 5^{2i+1} 7^{2a-2i-1})$$
is a factorization of $x^n$, with $\binom{k+1}{2} + 1$ distinct atoms, for some $n \le k^3$.  As such, we~conclude $\ell_0^M(x^n) \in \Theta(n^{2/3})$.  
\end{proof}

In view of the above, we pose the following question.  

\begin{question}\label{q:singularacmasymptotics}
Given a singular ACM $M_{a,b}$ and $x \in M_{a,b}$ with $x > 1$, determine the asymptotic behavior of $\ell_0^M(x^n)$, $\ell_1^m(x^n)$, and $\ell_\infty^m(x^n)$ as functions of $n$.  
\end{question}

\section*{Acknowledgements}

The authors were supported by NSF grant DMS-1851542, and made use of the software packages \cite{numericalsgpsgap,acmsage,numericalsgpssage} throughout this work.  


\end{document}